\documentclass[a4paper,11pt]{amsart}

\usepackage[top=2.5cm,bottom=2.7cm, outer=2.2cm, inner=2.2cm]{geometry}
\usepackage{amsfonts,amsmath,amssymb,bbm,amsthm}
\usepackage{enumitem}
\usepackage{color,pdflscape} 
\usepackage{hyperref} 

\DeclareMathOperator{\id}{id}

\renewcommand{\phi}{\varphi}
\renewcommand{\theta}{\vartheta}
\renewcommand{\epsilon}{\varepsilon}

\newcommand{\FF}{\mathbb{F}}

\newcommand{\NN}{\mathbb{N}}
\newcommand{\QQ}{\mathbb{Q}}

\newcommand{\cF}{\mathcal{F}}

\newcommand{\cI}{\mathcal{I}}

\newcommand{\cK}{\mathcal{K}}

\newcommand{\diff}{\mathrm{d}}
\newcommand{\dd}{\,\mathrm{d}}

\renewcommand{\theta}{\vartheta}

\newcommand{\1}{\mathbf{1}}

\newcommand*{\EX}[2][]{\E^{#1}\left [ #2 \right ]}
\newcommand*{\cEX}[3][]{\E^{#1}\left[ #2 \,\middle\vert\, #3 \right]}
\newcommand*{\PR}[2][\P]{{#1}\left [ #2 \right ]}
\newcommand*{\cPR}[3][\P]{{#1}\left [ #2 \;\middle\vert\; #3 \right ]}
\newcommand*{\as}[1]{#1\text{-a.s.}}

\newcommand*{\ol}[1]{\overline{#1}}
\newcommand*{\ul}[1]{\underline{#1}}

\usepackage{lipsum} 
\usepackage{color}

\usepackage{comment}

\newcommand{\R}{\mathbb R}

\newcommand{\E}{\mathbb E}
\newcommand{\N }{\mathbb N}
\newcommand{\q }{\mathbb Q}
\newcommand{\p }{\mathbb P}

\newcommand{\F }{\mathcal F}

\newcommand{\eps}{\varepsilon}

\renewcommand{\P}{\p}

\newtheorem{theorem}{Theorem}[section]
\newtheorem{corollary}[theorem]{Corollary}
\newtheorem{lemma}[theorem]{Lemma}
\newtheorem{proposition}[theorem]{Proposition}
\newtheorem{assumption}[theorem]{Assumption}

\theoremstyle{remark}
\newtheorem{remark}[theorem]{Remark}
\newtheorem{example}[theorem]{Example}
\newtheorem{definition}[theorem]{Definition}

\numberwithin{equation}{section}

\begin{document}
\title{Bubbles in discrete time models}

%
  \author{Martin Herdegen}
   \address{University of Warwick, United Kindom}
   \email{M.Herdegen@warwick.ac.uk}
   \author{D\"orte Kreher}
   \address{Humboldt-Universit\"at zu Berlin, Germany}
   \email{kreher@math.hu-berlin.de}

\begin{abstract}
	We introduce a new definition of bubbles in discrete-time models based on the discounted stock price losing mass at some finite drawdown under an equivalent martingale measure. We provide equivalent probabilistic characterisations of this definition and give examples of discrete-time martingales that are bubbles and those that are not.  In the Markovian case, we provide sufficient analytic conditions for the presence of bubbles. We also show that the existence of bubbles is directly linked to the existence of a non-trivial solution to a linear Volterra integral equation of the second kind involving the Markov kernel. Finally, we show that our definition of bubbles in discrete time is consistent with the strict local martingale definition of bubbles in continuous time in the sense that a properly discretised strict local martingale in continuous time is a bubble in discrete time.
\end{abstract}
	
 \subjclass[2010]{60G42, 60J10, 91G99, 45D05}

\keywords{Bubble, strict local martingale, Markov martingale, Volterra integral equation}
\thanks{We thank Johannes Ruf as well as Paolo Guasoni, the co-editor, and two anonymous referees for helpful comments on previous versions.}

\maketitle

Over the last 20 years, a mathematical theory of  bubbles for continuous time  models has been developed based on the concept of strict local martingales, cf.~the seminal papers by Loewenstein and Willard \cite{LoewensteinWillard}, Cox and Hobson \cite{CoxHobson}, Jarrrow, Protter and Shimbo \cite{JarrowProtter1,JarrowProtter2} as well as the survey article by Protter \cite{ProtterSurvey} and the references therein. In economic terms, an asset price bubble exists if the fundamental value of the asset deviates from its current price. If the fundamental value is understood to be the expectation of the (discounted) price process $S = (S_t)_{t \geq 0}$ under the equivalent local martingale measure $\p$, then the asset $S$ has a bubble if $\E^\p[S_T]<\E^\p[S_0]$ for some fixed time $T>0$, i.e., if $S$ is a strict local $\p$-martingale, that is a local $\P$-martingale that fails to be a $\P$-martingale.

While many implications and extensions of this definition have been discussed in the literature (see~e.g.~Ekstr\"om and Tysk \cite{EkstroemTysk}, Bayraktar, Kardaras and Xing \cite{Bayraktar},  Biagini, F\"ollmer and Nedelcu \cite{Biagini}, Herdegen and Schweizer \cite{HerdegenSchweizer}), the strict local martingale definition of  bubbles has no direct analogue in discrete time models. The reason for this is that a nonnegative local martingale $S = (S_k)_{k \in \NN_0}$ in discrete time with $S_0\in L^1$ is automatically a (true) martingale. Hence, a definition of bubbles based on strict local martingales is void. Also trying  to define a bubble in discrete time as a martingale $S = (S_k)_{k \in \NN_0}$ that is not uniformly integrable does not lead to a meaningful concept as this would imply that virtually all relevant models such as the standard binomial model (considered on an unbounded time horizon) are bubbles, which seems absurd.

Despite the above negative results, the goal of this paper is to introduce a new definition of  bubbles for discrete-time models on an unbounded time horizon -- keeping the standard assumption that the discounted stock price is a martingale. This definition has to satisfy at least two conditions.

\begin{itemize}
\item [(I)] It should \emph{split} martingales that are not uniformly integrable into two sufficiently rich classes: those that are  bubbles and those that are not. In particular, standard discrete-time models with i.i.d.~returns like the binomial model should not be  bubbles.
\item [(II)] It should be \emph{consistent} with the strict local martingale definition in continuous times in the sense that a properly discretised strict local martingale in continuous time is a  bubble in discrete time, and conversely, that a local martingale in continuous time is a strict local martingale if all properly discretised martingales are  bubbles in discrete time.
\end{itemize}

To the best of our knowledge, there has been no attempt in the extant literature to extend the martingale theory of  bubbles to discrete time. The only slight exception is Roch \cite{Roch} who introduced the notion of asymptotic asset price bubbles using the concept of weakly convergent discrete time models (“large financial market”). More precisely, he showed that even if the price process is a martingale in a sequence of weakly convergent discrete time models, it can have properties similar to a bubble in that the fundamental value in the asymptotic market can be lower than the current price in the asymptotic market. In contrast to \cite{Roch}, our approach is non-asymptotic. 

To motivate our definition of a  bubble in a discrete time model, consider a non-sophisticated investor who follows a simple buy-and-hold strategy to invest into an asset with (discounted) price process $S = (S_k)_{k \in \NN_0}$. The investor buys the asset at time $k=0$ and hopes that it will rise and rise. When the (discounted) asset price drops for the first time, the investor fears to lose money and sells the asset.
 Denoting by $\tau_1:=\inf\{j>0:\ S_j<S_{j-1}\}$ the time of the first drawdown of the asset, the fundamental value of $S$ at time $0$ (viewed with regard to the first drawdown) is $\E^\p[S_{\tau_1}]$, where $\P$ denotes an equivalent martingale measure. As the process $S$ is a nonnegative supermartingale, we always have $\E^\p[S_{\tau_1}] \leq S_0$. If $\E^\p[S_{\tau_1}] < S_0$, the fundamental value of the asset (viewed with regard to the first drawdown) is lower than its initial price and hence $S$ might be considered a  bubble. Indeed, if the market is complete, the predictable representation theorem implies that a sophisticated investor might choose a dynamic trading strategy $\theta$ whose (discounted) value process $V(\theta) = (V_k(\theta))_{k \in \NN_0}$ satisfies $V_0(\theta) = \E^\p[S_{\tau_1}] < S_0$ and $V_\tau(\theta) = S_\tau$.

Of course, the requirement that $S$ loses mass at the \emph{first} drawdown is somewhat arbitrary and unrealistic. For this reason, our precise definition of a  bubble in discrete time is more general and only requires $S$ to lose mass at the $k$-th drawdown for some $k \in \NN$. While this definition is very simple, it leads to a rich theory.

In Section \ref{ch:characterization}, we provide several equivalent probabilistic characterisations for a nonnegative discrete-time martingale to be a  bubble. We also provide necessary and sufficient  characterisations for a discrete-time martingale with independent increments to have a  bubble. In particular, we show that the standard binomial model does not have a  bubble, which implies that Condition (I) above is satisfied.

In Sections \ref{sec:Markov}, we look at the special case that $S$ is a Markov martingale. We provide characterisations for the presence or absence of  bubbles, depending on the \emph{probability} of going down, $a(x) = \p_x[S_1<x]$, and the \emph{relative recovery} when going down, $b(x) = \E_x[\tfrac{S_1}{x}\1_{\{S_1<x\}}]$. Loosely speaking, it turns out that  $S$ is a  bubble if and only if $b(x)$ converges to $0$ fast enough as $x \to \infty$. To make this precise, however, is quite involved. While we are able to give sufficient conditions in the general case, we provide necessary and sufficient conditions in the case of complete markets.

In Section \ref{ch:fixedpoint}, we continue our study of Markov martingales by looking more closely at the underlying Markov kernel. We show that the existence of  bubbles for $S$ is directly linked to the existence of a non-trivial nonnegative solution to a linear Volterra integral equation of the second kind involving the Markov kernel. Among others, this allows us to give some additional sufficient conditions for the existence of  bubbles that cannot be covered with the results from Section \ref{sec:Markov}.

Finally, in Section \ref{ch:relation}, we discuss how our definition of a  bubble in discrete time relates to the strict local martingale definition in continuous time. We show that when discretising a  positive continuous strict local martingale along sequences of stopping times in a certain somewhat canonical class, one obtains a  bubble in discrete time. Conversely, we show that a positive continuous local martingale is a strict local martingale if for all localising sequences in the same class, the corresponding discretised martingales are   bubbles. This shows that Condition (II) above is also satisfied.\footnote{As pointed out by one referee, our concept of a discrete-time bubble has a clear economic foundation which the strict local martingale definition of bubbles in continuous time maybe lacks -- at least at a first glance. Hence, by showing that appropriately discretised strict local martingales yield discrete-time bubbles, we provide additional and new support to modelling continuous-time bubbles by strict local martingales.} To prove these discretisation results we rely on the deep change of measure techniques first employed by Delbaen and Schachermayer \cite{DelbaenSchachermayer} and further developed by Pal and Protter \cite{PalProtter}, Kardaras, Kreher, and Nikeghbali \cite{KKN}, and Perkowski and Ruf \cite{PerkowskiRuf} that allow to turn the inverse of a nonnegative strict local martingale into a true martingale under a locally dominating probability measure. Some technical proofs of this section are shifted to the appendix.

\section{Definition and characterisation of bubbles}\label{ch:characterization}

In this section, we introduce our definition of a bubble in discrete time and provide equivalent probabilistic characterisations of this concept.

\begin{definition}\label{defbubble}
	Let $(\Omega, \cF, \FF = (\cF_k)_{k \in \NN_0},\P)$ be a filtered probability space. A nonnegative $(\P, \FF)$-martingale $S =(S_k)_{k\in\N_0}$ is called a \emph{bubble} if 
	\begin{equation}
		\label{eq:def:bubble}
		\EX{S_{\tau_{k}}} < \EX{S_0}
	\end{equation}
	for some $k\in\N$, where $\tau_0:=0$ and $\tau_k:=\inf\{j>\tau_{k-1}:\ S_j<S_{j-1}\},\ k\in\N$, denotes the $k$-th \emph{drawdown} of $S$. We also call $\P$ a \emph{bubble measure} for $S$.
\end{definition}

Some comments on the above definition are in order.
\begin{remark}
	(a) We include the possibility that $\P[\tau_k=\infty] > 0$ with positive probability. Since $S$ is a nonnegative martingale, it converges $\as{\P}$ to some integrable random variable $S_\infty$ by Doob's supermartingale convergence theorem. Hence, $S_{\tau_k}$ is well defined in any case.

	(b) If $S$ is a Markov process, $S$ will be a bubble if and only if $\E[S_{\tau_{1}}] < \E[S_0]$, cf.~Section \ref{sec:Markov} below. In general, however, the above definition does not contain any redundancy, thinking for example of dynamics with a change point. 
	
	(c) By the stopping theorem for uniformly integrable martingales, $S$ can only be a  bubble if it is {\it not} uniformly integrable. Example \ref{ex:non-bubble}, however, shows that not every martingale that fails to be uniformly integrable is a bubble. This can be compared to the strict local martingale definition in continuous time. Also, it emphasizes that the precise definition of the stopping time $\tau_k$ in \eqref{eq:def:bubble} is important and cannot naively be replaced by an arbitrary stopping time. 

(d) Our bubbles are strictly speaking \emph{$\p$-bubbles} since the definition depends on the choice of the (equivalent) martingale measure $\p$.  In incomplete markets, it is possible to have a \emph{$\p$-bubble} under some equivalent martingale measure (EMM) $\p$ but not a \emph{$\tilde \p$-bubble} under a different EMM $\tilde \p$. In this sense, our definition is not robust with respect to the choice of EMM. Note that exactly the same issue arises in the strict local martingale definition of bubbles in continuous time. Using similar ideas as in \cite{HerdegenSchweizer}, one could introduce the notion of a \emph{strong bubble} in discrete time. We leave the details of this to future work.
	
	(e) Definition \ref{defbubble} can be reformulated in the following way: For some $k \in \NN$, the stopped processes $S^{\tau_k}$ fails to be uniformly integrable. This reformulation allows to apply results from the literature that are proven for general c\`adl\`ag (local) martingales. For example,  Hulley and Ruf \cite{HulleyRuf} provide necessary and sufficient characterizations for a c\`adl\`ag (local) martingale to be uniformly integrable. However, the conditions given in \cite{HulleyRuf} require to calculate the distribution of the supremum of $S^{\tau_k}$ and the distribution of the jumps of $S^{\tau_k}$ at certain hitting times, both of which are difficult to get hold of. For this reason, they are not very useful in our discrete time setup.
\end{remark}

We proceed to give a first simple example of a bubble in a complete market model.

\begin{example}\label{ex:complete}
	Define the process $S = (S_k)_{k \in \NN_0}$ and the probability measure $\P$ recursively by $S_0 = s_0 > \tfrac{1}{2}$, and for $k \in \NN$,
	\begin{align*}
		&\cPR{S_k=2S_{k-1}-\frac{1}{2}}{S_{k-1}>\frac{1}{2}}=\cPR{S_k=\frac{1}{2}}{S_{k-1}>\frac{1}{2}}=\frac{1}{2} \\
&\cPR{S_k=\frac{1}{2}}{S_{k-1}=\frac{1}{2}}=1.
	\end{align*}
Then $S$ is a $\P$-martingale for its natural filtration, $\tau_1 < \infty$ $\as{\P}$, and $\E[S_{\tau_1}]=\frac{1}{2}<s_0= \E[S_0]$,
i.e., $S$ is a bubble.
\end{example}

The following result provides two equivalent characterisations of bubbles. The first one shows that a nonnegative martingale $S $ is a bubble if and only if there exists a deterministic time $k \in \NN_0$ such that $S$ loses mass at the first drawdown after $k$. The second one provides a limit characterisation. This second characterisation is particularly useful for checking whether or not a martingale $S$ is a bubble.

\begin{theorem}\label{thm:char:equivalent}
	Let $S = (S_k)_{k \in \NN_0}$ be a nonnegative martingale. For $k \in \NN_0$, define the stopping time $\tilde \tau_k$ by
	\begin{equation*}
		\tilde \tau_k := \inf\{j > k: S_j < S_{j-1}\}.
	\end{equation*}
	Then the following are equivalent:
	\begin{enumerate}
		\item $S$ is a bubble.
		\item There exists $k \geq 0 $ such that $\E[S_{\tilde \tau_k}] < \E[S_k]$.
		\item There exists $k \geq 0 $  such that $\lim_{n\to\infty}\E[(S_n-S_\infty)\1_{\{S_k\leq S_{k+1}\leq\dots\leq S_n\}}]>0$.
	\end{enumerate}
\end{theorem}

\begin{proof}
	(a) $\Rightarrow$ (b). Suppose that (b) is not true, i.e., $\E[S_{\tilde \tau_k}]=\E[S_k]$ for all $k\geq0$. Then for all $k\geq0$, the stopped process $S^{\tilde \tau_k}$ is a right-closed supermartingale which does not lose mass at $\infty$ and hence is uniformly integrable. We proceed by induction to show that $\E[S_{\tau_\ell}] =\E[S_0]$ for all $\ell \geq 0$, whence (a) is not true. The induction basis $\ell = 0$ is trivial. For the induction step, suppose that $\E[S_{\tau_{\ell-1}}]=\E[S_0]$ for some $\ell \geq 1$. Then using the stopping theorem for the uniformly integrable martingale $S^{\tilde \tau_k}$ in the third equality, we obtain
	\begin{align*}
		\E[S_{\tau_{\ell}}]&=\sum_{k=\ell-1}^\infty\E\left[S_{\tau_\ell}\1_{\{\tau_{\ell-1}=k\}}\right]+\E[S_\infty\1_{\{\tau_{\ell-1}=\infty\}}] \\
&=\sum_{k=\ell-1}^\infty\E\left[S_{\tilde \tau_k}\1_{\{\tau_{\ell-1}=k\}}\right]+\E[S_{\tau_{\ell-1}}\1_{\{\tau_{\ell-1}=\infty\}}]\\
		&=\sum_{k=\ell-1}^\infty\E\left[S_{\tau_{\ell-1}}\1_{\{\tau_{\ell-1}=k\}}\right]+\E[S_{\tau_{\ell-1}}\1_{\{\tau_{\ell-1}=\infty\}}]=\E[S_{\tau_{\ell-1}}]=\E[S_0].
	\end{align*}
	
	(b) $\Rightarrow$ (a). This follows from the fact that $\E[S_k] = \E[S_0]$ by the martingale property of $S$ together with the fact that $\tilde \tau_k \leq \tau_{k+1}$.
	
	(b) $\Leftrightarrow$ (c). The equivalence follows from the following calculation, which uses the martingale property of $S$ in the third equality:
	\begin{align}
		\EX{S_{\tilde \tau_k}} &=\lim_{n\to\infty}\sum_{\ell=k+1}^n\EX{S_\ell\1_{\{S_k\leq S_{k+1}\leq \dots\leq S_{\ell-1}>S_\ell\}}} + \EX{S_\infty\1_{\{S_k\leq S_{k+1}\leq \dots\}}}\notag \\
	&=\lim_{n\to\infty}\sum_{\ell=k+1}^n \left(\EX{S_\ell\1_{\{S_k\leq \dots\leq S_{\ell-1}\}}} - \EX{S_\ell\1_{\{S_k\leq \dots\leq S_{\ell}\}}} \right) \notag \\
&\qquad+ \EX{S_\infty\1_{\{S_k\leq S_{k+1}\leq \dots\}}}  \notag \\
	&=\lim_{n\to\infty}\sum_{\ell=k+1}^n \left(\EX{S_{\ell-1}\1_{\{S_k\leq \dots\leq S_{\ell-1}\}}} - \EX{S_\ell\1_{\{S_k \leq \dots\leq S_{\ell}\}}} \right) \notag \\
&\qquad + \EX{S_\infty\1_{\{S_k\leq S_{k+1}\leq \dots\}}}  \notag \\
		&= \EX{S_k} + \lim_{n\to\infty}\E[(S_\infty - S_n)\1_{\{S_k\leq S_{k+1}\leq\dots\leq S_n\}}]. \label{eq:pf:thm:char:equivalent}
\end{align}
	This finishes the proof.
\end{proof}

The characterisation (c) in Theorem \ref{thm:char:equivalent} is generally the most useful one  to decide whether or not $S$ is a bubble. The following corollary strengthens this characterisation. It shows directly that bounded martingales fail to be bubbles. (Of course, this follows directly from the fact that a bounded martingale is uniformly integrable.)

\begin{corollary}
	\label{cor:char:equivalent}
	Let $S = (S_k)_{k \in \NN_0}$ be a nonnegative martingale.
	Then the following are equivalent:
	\begin{enumerate}
		\item $S$ is a bubble.
		\item There exist $x \geq 0$ and $k \geq 0 $ such that $$\lim_{n\to\infty}\E[(S_n-S_\infty)\1_{\{x \leq S_k\leq S_{k+1}\leq\dots\leq S_n\}}]>0.$$
		\item For all $x \geq 0$, there exists $k \geq 0 $  such that $$\lim_{n\to\infty}\E[(S_n-S_\infty)\1_{\{x \leq S_k\leq S_{k+1}\leq\dots\leq S_n\}}]>0.$$
	\end{enumerate}
\end{corollary}

\begin{proof}
	It is clear that (c) implies (b), and (b) a fortiori implies condition (c) of Theorem \ref{thm:char:equivalent}, which by Theorem \ref{thm:char:equivalent} gives (a). So it remains to prove (a) $\Rightarrow$ (c).
	So fix $x \geq 0$. Since $S$ is a bubble, by Theorem \ref{thm:char:equivalent}(b) and the martingale property of $S$, there exists $\ell\geq 0$ such that $\E[S_{\tau^\ell_1}] < \E[S_\ell] = \E[S_0]$.
	
	Define the stopping times $\sigma_{\ell,x}$ and $\tau^{\sigma_{\ell,x}}_1$ by
	\begin{align*}
		\sigma_{\ell,x} := \inf\{k \geq \ell : S_k \geq x\} \quad \text{and} \quad \tau^{\sigma_{\ell,x}}_1 := \inf\{j > \sigma_{\ell,x}: S_j < S_{j-1}\}.
	\end{align*}
Then $\sigma_{\ell,x}$ is the first hitting time of $[x, \infty)$ after $\ell$ and $\tau^{\sigma_{\ell,x}}_1$ the first drawdown of $S$ after $\sigma_{\ell,x}$. Since $\sigma_{\ell,x} \geq  \ell$, it follows that $\tau^{\sigma_{\ell,x}}_1 \geq \tau^\ell_1$.

By the definition of $\sigma_{\ell,x}$, the stopped process $S^{\sigma_{\ell,x}}$ is uniformly integrable because 
	$$S^{\sigma_{\ell,x}}_k \leq \max(S_0, \ldots, S_\ell, S_{\sigma_{\ell,x}}, x)\in L^1.$$
	This implies that $\E[S_{\sigma_{\ell,x}}] = \E[S_0]$.
	Since $\tau^{\sigma_{\ell,x}}_1 \geq \tau_1^\ell$, this  in turn implies that
	\begin{equation*}
		\EX{S_{\tau^{\sigma_{\ell,x}}_1}} \leq \EX{S_{\tau^\ell_1}} < \EX{S_0} = \EX{S_{\sigma_{\ell,x}}}.
	\end{equation*}
	A similar calculation as in \eqref{eq:pf:thm:char:equivalent} shows that 
	\begin{equation*}
		\E\big[S_{\tau^{\sigma_{\ell,x}}_1}\big]  = \EX{S_{\sigma_{\ell,x}}} + \lim_{n\to\infty}\E\Big[(S_\infty - S_{\sigma_{\ell,x}+n})\1_{\{S_{\sigma_{\ell,x}}\leq\dots\leq S_{\sigma_{\ell,x} + n}\}}\1_{\{\sigma_{\ell,x} < \infty\}}\Big].
	\end{equation*}
	This together with the tower property of conditional expectations and dominated convergence yields
	\begin{equation*}
		\EX{\lim_{n\to\infty} \cEX{(S_{\sigma_{\ell,x}+n }- S_\infty)\1_{\{S_{\sigma_{\ell,x}}\leq S_{\sigma_{\ell,x} +1}\leq\dots\leq S_{\sigma_{\ell,x} + n}\}}}{\cF_{\sigma_{\ell,x}}} \1_{\{\sigma_{\ell,x} < \infty\}}} > 0.
	\end{equation*}
	We may deduce that there is $k \geq \ell$ such that
	\begin{equation*}
		\EX{\lim_{n\to\infty} \cEX{(S_{k +n} - S_\infty)\1_{\{S_{k}\leq S_{k+1}\leq\dots\leq S_{k + n}\}}\1_{\{\sigma_{\ell,x} = k\}}}{\cF_k} } > 0.
	\end{equation*}
	Using that $S_k \geq x$ on $\{\sigma_{\ell,x} = k\}$, this implies that 
	\begin{equation*}
		\EX{\lim_{n\to\infty} \cEX{(S_{k +n} - S_\infty)\1_{\{x \leq S_{k}\leq S_{k+1}\leq\dots\leq S_{k + n}\}}}{\cF_k} } > 0.
	\end{equation*}
	Now dominated convergence and the tower property of conditional expectations give (c).
\end{proof}

While characterisations (b) or (c) in Corollary \ref{cor:char:equivalent} are an improvement of Theorem \ref{thm:char:equivalent}(c), they still depend on $S_\infty$, of which we generally do not have a good knowledge. The following corollary provides a mild condition on $S$ under which the  bubble behaviour of $S$ can be characterised without involving $S_\infty$.

\begin{proposition}
	\label{prop:BC}
	Let $S = (S_k)_{k \in \NN_0}$ be a nonnegative martingale and $x  > 0$. Suppose that
	\begin{equation}
		\label{eq:prop:BC}
		\sum_{k = 0}^\infty \cPR{S_k < x\text{ or } S_k > S_{k+1}}{\cF_k}  = \infty \;\; \as{\P}
	\end{equation}
	Then for each $k \in \NN_0$,
		\begin{equation}
			\label{eq:prop:BC:infty}
		\lim_{n\to\infty}\EX{S_\infty \1_{\{x \leq S_k\leq S_{k+1}\leq\dots\leq S_n\}}}  = 0.
	\end{equation}
	In particular, $S$ is a bubble if and only if there exists $k \in \NN_0$ such that
	\begin{equation*}
		\lim_{n\to\infty}\E[S_n\1_{\{x \leq S_k\leq S_{k+1}\leq\dots\leq S_n\}}] > 0.
	\end{equation*}
\end{proposition}

\begin{proof}
The conditional Lemma of Borel--Cantelli implies that
	\begin{equation*}
		\PR{\liminf_{k \to \infty} \{x  \leq S_k \leq S_{k+1}\}} = 1- \PR{\limsup_{k \to \infty} \left(\{S_k < x\} \cup \{S_k > S_{k+1}\}\right)}=0.
	\end{equation*}
	This implies a fortiori that $\{x  \leq S_k \leq S_{k +1} \leq  \cdots \}$ is a $\P$-null set for each $k \in \NN_0$. This gives \eqref{eq:prop:BC:infty}. The final claim now follows from Corollary \ref{cor:char:equivalent}.
\end{proof}

The following result gives simple necessary and sufficient conditions for a nonnegative martingales with independent increments to be a bubble.

\begin{theorem}
	\label{thmr:indep:incr}
	Let $(X_k)_{k \in \NN}$ be an independent sequence of nonnegative random variables with $\E[X_k] = 1$ for $k \in \NN$. Define the process $S = (S_k)_{k \in \NN_0}$ by $S_k=\prod_{\ell =1}^k X_\ell$ and the filtration $\FF = (\cF_k)_{k \in \NN_0}$ by $\cF_k := \sigma(S_0, \ldots, S_k)$. Moreover, for $k \in \NN$ set
	\begin{equation*}
		a_k := \P[X_k <1] \in [0, 1), \quad b_k:=\EX{X_k\1_{\{X_k < 1\}}} \in [0, a_k].
	\end{equation*}
	Then $S$ is a positive $(\P, \FF)$-martingale.  It is a bubble if and only if
	\begin{equation*}
		\sum_{k =1}^\infty a_k =\infty \quad \text{and} \quad \sum_{k =1}^\infty b_k < \infty.
	\end{equation*}
\end{theorem}

\begin{proof}
	It is clear by construction that $S$ is a nonnegative $(\P, \FF)$-martingale. 
	
	First, we argue that if $\sum_{k =1}^\infty a_k < \infty$, then $S$ is uniformly integrable and hence cannot be a  bubble. To this end, note that
	\[\infty>\sum_{k=0}^\infty a_k=\sum_{k=0}^\infty\p[X_k<1]\geq \sum_{k=0}^\infty \E\big[(1-\sqrt{X_k})\1_{\{X_k<1\}}\big]\geq \sum_{k=0}^\infty \E\big[(1-\sqrt{X_k})\big].\]
By Kakutani's theorem (see, e.g.~\cite[Theorem 14.12(v)]{Williams}) this implies that $S$ is uniformly integrable.

	Next, if $\sum_{k =1}^\infty a_k = \infty$, by independence of the $X_k$,
	\begin{equation*}
		\sum_{k = 0}^\infty \cPR{S_k < 1\text{ or } S_k > S_{k+1}}{\cF_k} \geq  \sum_{k = 0}^\infty \cPR{X_{k+1} < 1}{\cF_k} = \sum_{k =1}^\infty a_k = \infty.
	\end{equation*}
	This together with Corollary \ref{cor:char:equivalent} implies that $S$ is a bubble if and only if there exists $k \geq 0$ such that 
	\begin{align*}
		\lim_{n\to\infty}\E[S_n\1_{\{1 \leq S_k\leq S_{k+1}\leq\dots\leq S_n\}}] &=  \P[ S_k \geq 1] \prod_{\ell =k+1}^\infty   \EX{X_k\1_{\{X_k \geq 1\}}} \\
&= P[ S_k \geq \epsilon] \prod_{\ell =k+1}^\infty (1 - b_\ell)  > 0.
	\end{align*}
	Since $\P[ S_k \geq 1] > 0$ and $(1- b_\ell) > 0$ for all $\ell \in \NN$, this is equivalent to $\prod_{\ell =1}^\infty (1 - b_\ell) > 0$, which in turn is equivalent to $\sum_{k =1}^\infty b_k < \infty$.
\end{proof}

We illustrate the above theorem by two examples. The first one gives  bubble in a time-dependent binomial-type model, where the downward jumps get more and more severe. 
\begin{example}
	\label{ex:bubble}
	Let $(X_k)_{k \in \NN}$ be a sequence of independent random variables with $P[X_k = \frac{1}{k}] = \frac{1}{k}$ and $P[X_k = 1 + \frac{1}{k}] = 1- \frac{1}{k}$. Define the process $S = (S_k)_{k \in \NN_0}$ by $S_k:=\prod_{\ell =1}^k X_\ell$  and the filtration $\FF = (\cF_k)_{k \in \NN_0}$ by ${\cF_k := \sigma(S_0, \ldots, S_k)}$. Then $a_k = \frac{1}{k}$ and $b_k = \frac{1}{k^2}$. Hence, $\sum_{k =1}^\infty a_k =\infty$ and $\sum_{k =1}^\infty b_k < \infty$, whence $S$ is a  bubble.
\end{example}

The second example shows that a martingale with i.i.d.~returns is never a bubble. In particular an standard binomial model is never a bubble, which coincides with our intuition.
\begin{example}
	\label{ex:non-bubble}
	Let $(X_k)_{k \in \NN}$ be a sequence of i.i.d.~random variables that are nonnegative and satisfy $\E[X_k] = 1$ and $\P[X_k \neq 1] > 0$. Define the process $S = (S_k)_{k \in \NN_0}$ by $S_k:=\prod_{\ell =1}^k X_\ell$ and the filtration $\FF = (\cF_k)_{k \in \NN_0}$ by ${\cF_k := \sigma(S_0, \ldots, S_k)}$. Then $S$ is a martingale but fails to be uniformly integrable because $S_n=\exp\left(\sum_{k=1}^n\log(X_k)\right)\stackrel{a.s.}{\longrightarrow}0$ by the strong law of large numbers since $\E[\log(X_k)] < 1$. However, setting $b_k := \E[X_k\1_{\{X_k < 1\}}] = b > 0$, we obtain
	\begin{equation*}
		\sum_{k = 1}^\infty b_k = \infty.
	\end{equation*}
	Thus, $S$ fails to be a bubble.
\end{example}


\section{Characterization of bubble measures for Markov chains}\label{sec:Markov}

Throughout this section, we suppose that $S=(S_k)_{k \in \NN_0}$ is a positive Markov martingale with transition kernel $K:(0,\infty)\times\mathcal{B}_{(0,\infty)} \to [0,\infty)$, starting from $S_0=x>0$. Our goal is to determine under which conditions on the kernel $K(x,  \dd y)$ the measure $\p_{x}$ is a bubble measure.  To this end, we define the functions $a, b: (0, \infty) \to [0, 1)$ by
\begin{align}
	a(x) &:= \P_x\left[S_1  <  x\right] = \int_{[0, x)} K(x, \diff y), \label{eq:Markov:a}\\
	b(x) &:= \E_x\left[\frac{S_1}{x} \1_{\{S_1 < x\}}\right]= \int_{[0, x)} \frac{y}{x} K(x, \diff y). \label{eq:Markov:b}
\end{align}
Here, $a(x)$ denotes the \emph{probability} of a downward jump and $b(x)$ the \emph{relative recovery} in case of a downward jump.

First, we show we show that $S$ cannot be a bubble unless the relative recovery function $b$ converges to zero at infinity.

\begin{proposition}
	\label{prop:Markov:non bubble}
Assume that $\liminf_{x \to \infty} b(x) > 0$. Then $S$ fails to be a bubble under $\P_{x}$ for any $x \in (0, \infty)$.
\end{proposition}

\begin{proof}
	There exists $x_0 > 0$ and $\epsilon \in(0,1)$ such that $b(x) \geq \epsilon$ for all $x \geq x_0$. Pick $x \in (0, \infty)$. By Corollary \ref{cor:char:equivalent} and using that $S_\infty \geq 0$, it suffices to show that for each $k \in \NN_0$,
	\begin{equation*}
		\lim_{n \to \infty} \E_{x} \left[S_n \1_{\{x_0 \leq S_k \leq S_{k+1} \leq \cdots \leq S_n\}}\right] = 0.
	\end{equation*}
	So let $k < n$. Then by the Markov property and the definition of $b$,
	\begin{align}
		&\E_{x} \left[S_n \1_{\{x_0 \leq S_k \leq S_{k+1} \leq \cdots \leq S_n\}}\right] \notag \\
&\quad = \E_{x} \left[ \E_{x} \left[\frac{S_n}{S_{n-1}} \1_{\{S_n \geq S_{n-1}\}}\,\middle |\, \cF_{n-1} \right] S_{n-1} \1_{\{x_0 \leq S_k \leq S_{k+1} \leq \cdots \leq S_{n-1}\}}\right] \notag \\
		&\quad= \E_{x} \left[ (1- b(S_{n-1})) S_{n-1} \1_{\{x_0 \leq S_k \leq S_{k+1} \leq \cdots \leq S_{n-1}\}}\right] \notag \\
		&\quad\leq (1- \epsilon) \E_{x} \left[S_{n-1} \1_{\{x_0 \leq S_k \leq S_{k+1} \leq \cdots \leq S_{n-1}\}}\right] \notag \\
		&\quad\leq (1- \epsilon)^{n-k} \E_{x}\big[S_k\1_{\{x_0 \leq S_k\}}\big] \leq  (1- \epsilon)^{n-k}x.
	\end{align}
	Now the claim follows by letting $n \to \infty$.
\end{proof}

We proceed to formulate a mild condition on the function $a$, which allows us to characterize the bubble behaviour of $S$  without involving $S_\infty$.

\begin{assumption}\label{ass:lowerbound:a}
There exists $x_a > 0$ such that  $\inf_{x \in [x_a, y]} a(x) > 0$ for all $y > x_a$.  
\end{assumption}

\begin{remark}
	\label{rem:lowerbound:a}
	Assumption \ref{ass:lowerbound:a} is in particular fulfilled if $a$ is positive and lower semi-continuous.
\end{remark}

\begin{proposition}
	\label{prop:Markov:cond a}
Suppose Assumption \ref{ass:lowerbound:a} is satisfied for some $x_a>0$. Let $x' \geq x_a$. Then for each $k \in \NN_0$,
		\begin{equation}
	\label{eq:prop:Markov:cond a}
	\lim_{n\to\infty}\E_{x} \left[S_\infty \1_{\{x' \leq S_k\leq S_{k+1}\leq\dots\leq S_n\}}\right]  = 0.
\end{equation}
Moreover, $S$ is a bubble under $\P_{x}$ if and only if there exists $k \in \NN_0$ such that
	\begin{equation*}
		\lim_{n \to \infty} \E_{x} \left[S_n \1_{\{x' \leq S_k \leq S_{k+1} \leq \cdots \leq S_n\}}\right] > 0.
	\end{equation*}
\end{proposition}

\begin{proof}
	By Proposition \ref{prop:BC}, it suffices to check that $\as{\P_{x}}$
	\begin{equation*}
		\sum_{k = 0}^\infty \P_{x}\left[S_k < x'\text{ or } S_k > S_{k+1}\, \middle|\, \cF_k \right] =  \sum_{k = 0}^\infty (\1_{\{S_k < x'\}} + a(S_k)\1_{\{S_k \geq x'\}} )= \infty.
	\end{equation*}
We now distinguish two cases: If $\omega \in \limsup_{k \to \infty} \{S_k < x'\}$, it follows that $\sum_{k = 0}^\infty \1_{\{S_k(\omega) < x'\}}= \infty$. If $\omega \in \liminf_{k \to \infty} \{S_k \geq x'\} \cap \{\sup_{k \geq 0} S_k < \infty\}$, then $\sum_{k = 0}^\infty  a(S_k(\omega))\1_{\{S_k(\omega) \geq x\}} ) = \infty$ by the assumption on $a$. Since $\sup_{k \geq 0} S_k < \infty$ $\as{\P_{x}}$ by Doob's martingale convergence theorem, the claim follows.
\end{proof}

We next aim to give a sufficient characterisation for $S$ being a bubble. To this end, we need to slightly relax the definition of the function $b$ from \eqref{eq:Markov:b}: 
For $\epsilon > 0$, define the function $b_\epsilon : (0, \infty) \to (0, 1]$ by
\begin{equation*}
	b_\epsilon(x) := \E_x\left[\frac{S_1}{x} \1_{\{S_1  < x (1+ \epsilon)\}}\right]= \int_{[0, x(1 +\epsilon))} \frac{y}{x} K(x, \diff y).
\end{equation*}

\begin{theorem}
	\label{thm:Markov:sufficient:bubble}
Suppose Assumption \ref{ass:lowerbound:a} is satisfied and there exists $\epsilon > 0$ and $x_b > 0$ such that the function $b_\epsilon$ is nonincreasing for $x \geq x_b$ and satisfies $\int_{\log(x_b)}^\infty b_\epsilon(\exp(x)) \dd x < \infty$. Then $S$ is a bubble under each $\P_{x}$, for which $S$ is not $\P_{x}$-a.s.~bounded.
\end{theorem}

\begin{proof}
	Suppose that $S$ is not $\P_{x}$-a.s.~bounded. Let $x_a$ be the constant in Assumption \ref{ass:lowerbound:a}. We may assume without loss of generality that $x_b \geq x_a$. By Proposition \ref{prop:Markov:cond a}, it suffices to check that there is $k \in \NN_0$ such that
	\begin{align*}
		&\lim_{n \to \infty} \E_{x} \left[S_n \1_{\{x_b\leq S_k \leq S_{k+1} \leq \cdots \leq S_{k + n}\}}\right] \\
&\quad\geq 	\lim_{n \to \infty} \E_{x0} \bigg[S_k \1_{\{S_k \geq x_b\}} \prod_{j =1}^n \frac{S_{k+j}}{S_{k+j-1}}\1_{\{S_{k+j} \geq S_{k+j-1}(1 + \epsilon)\}}  \bigg]  > 0.
	\end{align*}
	Since $S$ is not $\P_{x}$-a.s.~bounded, there exists $k \in \NN_0$ with $\E_{x}[S_k\1_{\{S_k \geq x_b\}}] > 0$. Fix $n \in \NN$. By the definition of  $b_\epsilon$ , we obtain for $j \in \{1, \ldots n\}$
	\begin{align*}
		&\E_{x} \bigg[\frac{S_{k+j}}{S_{k+j-1}}\1_{\{S_{j+k-1}(1 + \epsilon) \leq S_{j+k}\}} \,\bigg|\, \cF_{k+j-1} \bigg] = (1 - b_\epsilon(S_{k+j-1}))\;\; \as{\P_{x}}
	\end{align*}
	This together with the tower property of conditional expectations and the fact that $b_\epsilon$ is nonincreasing for $x \geq x_b$ gives
	\begin{align*}
		&\E_{x} \bigg[S_k \1_{\{S_k \geq x_b\}} \prod_{j =1}^n \frac{S_{k+j}}{S_{k+j-1}}\1_{\{S_{k+j} \geq S_{k+j-1}(1 + \epsilon)\}}  \bigg]  \\
		&= \E_{x} \bigg[S_k \1_{\{S_k \geq x_b\}} \bigg(\prod_{j =1}^{n-1} \frac{S_{k+j}}{S_{k+j-1}}\1_{\{S_{k+j} \geq S_{k+j-1}(1 + \epsilon)\}}\bigg) (1 - b_\epsilon(S_{k+n-1}))  \bigg] \\ 
		& \geq \E_{x} \bigg[S_k \1_{\{S_k \geq x_b\}} \bigg(\prod_{j =1}^{n-1} \frac{S_{k+j}}{S_{k+j-1}}\1_{\{S_{k+j} \geq S_{k+j-1}(1 + \epsilon)\}}\bigg) (1 - b_\epsilon(x_b(1 +  \epsilon)^{n-1})  \bigg]  \\
		&\geq \E_{x} \bigg[S_k \1_{\{S_k \geq x_b\}} \prod_{j =0}^{n-1}  (1 - b_\epsilon(x_b(1 +\epsilon)^j)) \bigg] \\
&= \E_{x} \left[S_k \1_{\{S_k \geq x_b\}}\right]\prod_{j =0}^{n-1}  (1 - b_\epsilon(x_b(1 + \epsilon)^j)). 
	\end{align*}
	Thus, it remains to show that $\prod_{j =0}^\infty  (1 - b_\epsilon(x_b (1 + \epsilon)^j)) > 0$. The latter condition is equivalent to $\sum_{j =0}^\infty b_\epsilon(\exp(\log (x_b) + \log(1 + \epsilon) j )) < \infty$, which is equivalent to $\int_{\log(x_b)}^\infty b_\epsilon(\exp(x)) \dd x < \infty$.
\end{proof}

We illustrate the above result first by two examples. The first one is a  “smooth” version of Example \ref{ex:complete}; the second one is an example of a “discrete diffusion” for the log price.
\begin{example}
Assume that the Markov kernel is given by
	\begin{equation*}
		K(x,\diff y)=\begin{cases}  \frac{1}{2} \1_{(0, 1)}(y)  \dd y + \frac{1}{2}\1_{(2x-1, 2x)}(y) \dd y&\text{if } x> 1,\\
			 \frac{1}{2x} \1_{(0, 2x)}(y)  \dd y &\text{if } x \leq 1. \end{cases}
	\end{equation*}
	Then $a(x) = \frac{1}{2} $ and for $\epsilon \in (0, 1)$,
$$b_\epsilon(x) = \tfrac{(1+\epsilon)^2}{4}\1_{\{x \leq 1\}} + (1 - x(1 - \tfrac{(1+\epsilon)^2}{4}))\1_{\{1 \leq x \leq  \tfrac{1}{1 -\epsilon}\}}  +\tfrac{1}{4 x} \1_{\{x >  \tfrac{1}{1 -\epsilon}\}} .$$ By Theorem \ref{thm:Markov:sufficient:bubble}, $S$ is a bubble under $\P_{x}$ for all $x>0$. 
\end{example}

\begin{example}
\label{ex:Gaussian}
Let $(Z_k)_{k\in\N}$ be a sequence of i.i.d.~standard normal random variables and $\sigma:\R\to(0,\infty)$ a measurable function such that $\sigma(x)$ is nondecreasing for large values of $x$. Define the process $(X_k)_{k \in \NN_0}$ recursively by $X_0 := 0$ and 
\begin{equation*}
X_{k+1}:=X_k+\sigma(X_k)Z_k-\frac{\sigma^2(X_k)}{2},\quad k\in\N_0.
\end{equation*}
Then the process $S=(S_k)_{k\in\N_0}$ defined by $S_k :=\exp(X_k)$ for $k \in \NN_0$ is a Markov martingale.

Denoting by $\Phi$ the cdf of a standard normal random variable, it is not difficult to check that for $x > 0$ and $\epsilon \geq 0$, 
\begin{equation*}
a(x)=\Phi\left(\frac{\sigma(\log(x))}{2}\right),\qquad b_\eps(x)=\Phi\left(\frac{\log(1+\eps)}{\sigma(\log(x))}-\frac{\sigma(\log(x))}{2}\right)
\end{equation*}
where $b_0 = b$.
Thus, by Proposition \ref{prop:Markov:non bubble}, for $S$ to be a bubble it is necessary that $\sigma(x)\to\infty$ as $x\to\infty$.
Moreover, by Theorem \ref{thm:Markov:sufficient:bubble}, a sufficient condition for $S$ to be a bubble is given by 
\begin{equation}
\label{eq:ex:Gaussian}
\int^\infty_{x_0} \Phi\left(\frac{\log(2)}{\sigma(x)}-\frac{\sigma(x)}{2}\right) \dd x<\infty\quad\text{for some }x_0\in\R.
\end{equation}
Denoting the pdf of a standard normal random variable by $\phi$, using Mill's ratio and the fact that $|(\frac{\log(2)}{\sigma(x)}-\frac{\sigma(x)}{2})^2 - (\frac{\sigma(x)}{2})^2 | \leq (\frac{\log(2)}{\sigma(x)})^2+\log(2)$ is uniformly bounded for all $x$ sufficiently large as $\sigma$ is nondecreasing, it is not difficult to check that \eqref{eq:ex:Gaussian} is equivalent to 
\begin{equation}\label{ex:eq:Gaussian2}
\int^\infty_{x_0} \frac{1}{\sigma(x)}\varphi\left(-\frac{\sigma(x)}{2}\right)\dd x<\infty\quad\text{for some }x_0\in\R.
\end{equation}
\end{example}

\begin{remark}
It is insightful to compare Example \ref{ex:Gaussian} to the continuous-time theory of bubbles. To this end, 
 recall that in continuous time, the process $S = (S_t)_{t \geq 0}$ given by $S_t:=\exp(X_t)$, where $X=(X_t)_{t\geq0}$ solves the SDE 
$$dX_t=\sigma(X_t)dW_t-\frac{1}{2}\sigma^2(X_t)dt,$$ 
is a strict local martingale and hence a bubble if and only if for some $x_0\in\R$,
\begin{equation}\label{ex:eq:Gaussian3}
 \int^\infty_{x_0}\frac{dx}{\sigma^2(x)}<\infty,
\end{equation}
cf.~\cite[Corollary 4.3]{MijatovicUrusov}. While \eqref{ex:eq:Gaussian2} and \eqref{ex:eq:Gaussian3} both say that $S$ is a bubble if and only if $\sigma(x)\to\infty$ fast enough as $x\to\infty$, the exact rate of increase of $\sigma$ required for a bubble is quite different as \eqref{ex:eq:Gaussian3} is a much stronger requirement than \eqref{ex:eq:Gaussian2} on the growth of $\sigma$. The reason for this is that the discretisation of the diffusion model in continuous time should not be done along a deterministic time grid, but along certain sequences of stopping times; see Section \ref{ch:relation} below.
\end{remark}

While Proposition \ref{prop:Markov:non bubble} and Theorem \ref{thm:Markov:sufficient:bubble} give useful general sufficient conditions for the absence or presence of a bubble, respectively, these conditions are not necessary. In the complete Markov case, we can give a necessary and sufficient characterisation of bubbles under mild assumptions on the functions $a$ and $b$.

\begin{theorem}
	\label{thm:Markov:complete}
Suppose the Markov kernel is given by
	\begin{equation*}
		K(x, \diff y) = a(x)\delta_{\frac{b(x)x}{a(x)}} (\diff y) + (1- a(x))\delta_{\frac{(1-b(x))x}{1-a(x)}} (\diff y),
	\end{equation*}
	where $0 \leq b(x) < a(x) < 1$ and  $0 < \liminf_{x \to \infty} a(x)  \leq \limsup_{x \to \infty} a(x) < 1$.  Moreover, suppose that there exists $x_b > 0$ such that the function $b$ is nonincreasing for $x \geq x_b$. Then $S$ is a bubble under $\P_{x}$ if and only if $S$ is not $\P_{x}$-a.s.~bounded and $\int_{\log(x_b)}^\infty b(\exp(x)) \dd x < \infty$.
\end{theorem}

\begin{proof}
	Suppose that $S$ is not $\P_{x}$-a.s.~bounded. We may focus on the case that  $\lim_{x \to \infty} b(x) = 0$. Indeed, otherwise it follows that $\lim_{x \to \infty} b(x) > 0$
	and hence $\int_{\log(x_b)}^\infty b(\exp(x)) \dd x = \infty$ and $S$ fails to be a bubble by Proposition \ref{prop:Markov:non bubble}.
	
Since $\lim_{x \to \infty} b(x) = 0$ and  $0 < \liminf_{x \to \infty} a(x)  \leq \limsup_{x \to \infty} a(x) < 1$, after potentially enlarging $x_b$, we may assume that there exists $0 < c < C$ such that $1 +c  \leq \frac{1- b(x)}{1 - a(x)} \leq 1 + C$ for all $x \geq  x_b$. By Proposition \ref{prop:Markov:cond a}, $S$ is a bubble under $\P_{x}$ if and only if there exists $k \geq 0$ such that
	\begin{align*}
		&\lim_{n \to \infty} \E_{x} \left[S_n \1_{\{x_b\leq S_k \leq S_{k+1} \leq \cdots \leq S_{k + n}\}}\right] >  0.
	\end{align*}
Using that $\{S_j \leq S_{j+1}\} = \{S_{j+1} = \frac{1- b(S_j)}{1 - a(S_j)} S_j\}$ for $j \in \NN_0$ and arguing as in the proof of Theorem \ref{thm:Markov:sufficient:bubble} gives for each $0 \leq k \leq n$,
	\begin{align*}
		&\E_{x} \left[S_k \1_{\{S_k \geq x_b\}}\right] \prod_{j =0}^{n-1}  (1 - b(x_b(1 + C)^j)) \geq \E_{x} 	\left[S_n \1_{\{x_b\leq S_k \leq S_{k+1} \leq \cdots \leq S_{k + n}\}}\right] \\
		&\qquad\qquad\qquad\geq \E_{x} \left[S_k \1_{\{S_k \geq x_b\}}\right] \prod_{j =0}^{n-1}  (1 - b(x_b(1 + c)^j)). 
	\end{align*}
	Now using that for $\gamma \in \{c, C\}$, $\prod_{j =0}^\infty  (1 - b(x_b (1 + \gamma)^j) > 0$ if and only if $\sum_{j =0}^\infty b(\exp(\log (x_b) + \log(1 + \gamma) j )) < \infty$, which in turn is equivalent to $\int_{\log(x_b)}^\infty b(\exp(x)) \dd x < \infty$, the claim follows.
\end{proof}


\section{A fixed point equation associated to a Markovian bubble}\label{ch:fixedpoint}
In this section, we continue our study of Markov martingales, taking a more analytic perspective. We assume throughout that $S=(S_k)_{k \in \NN_0}$ is a positive Markov martingale with kernel $K:(0,\infty)\times\mathcal{B}_{(0, \infty)}\to[0,\infty)$, starting from $S_0=x>0$. The key object of this section is the \emph{default function} of~$S$.

\begin{definition}
	\label{def:default}
The Borel-measurable function $M_S: (0, \infty) \to [0, \infty)$, defined by
\begin{equation}\label{M}
	M_S(x):=\lim_{n\to\infty}\E_x\left[(S_n-S_\infty)\1_{\{S_n\geq S_{n-1}\geq \dots\geq S_1\geq x\}}\right],
\end{equation}
is called the \emph{default function} of $S$.
\end{definition}

It follows from \eqref{eq:pf:thm:char:equivalent} (with $k = 0$) that $M_S(x) =\E_x[S_{\tau_1}]$, so that $M_S$ measures the loss of mass at the first drawdown of $S$. It is clear that $\P_{x}$ is a bubble measure for $S$ if $M_S(x) > 0$. The following two results show that $M$ essentially fully characterises the bubble behaviour of $S$ under $\P_x$ for all $x > 0$.

\begin{proposition}
	\label{prop:MS:zero}
Suppose $M_S(x) = 0$ for all $x \geq x' > 0$. Then $M_S(x) = 0$ for all $x > 0$, and $\P_x$ fails to be a bubble measure for $S$ for any $x > 0$.
\end{proposition}

\begin{proof}
Fix $x > 0$. Then for all $k \in \NN_0$, the Markov property of $S$, the choice of $x'$ and dominated convergence give
\begin{align*}
	&\lim_{n\to\infty}\E_x[(S_n-S_\infty)\1_{\{x' \leq S_k\leq S_{k+1}\leq\dots\leq S_n\}}] \\
&\qquad= \E_{x}\left[\lim_{n \to \infty} \E_{S_k}[(S_n-S_\infty)\1_{\{S_k\leq S_{k+1}\leq\dots\leq S_n\}}] \1_{\{S_k \geq x'\}}\right] \\
&\qquad= \E_x\left[M_S(S_k)\1_{\{S_k \geq x'\}}\right] = 0.
\end{align*}
Thus, Corollary \ref{cor:char:equivalent} implies that $S$ is not a bubble under $\P_x$ and so $M_S(x) = 0$.
\end{proof}

\begin{proposition}
	\label{prop:MS:liminf}
	Suppose that $\liminf_{x \to \infty} M_S(x) > 0$. Then $S$ is a bubble under $\P_x$ for all $x > 0$ for which $S$ is not $\P_x$-a.s. bounded.
\end{proposition}

\begin{proof}
	Fix $x > 0$ and suppose that $S$ is not $\P_x$-a.s. bounded. By hypothesis, there exists $x' \geq x$ such that $M_S(y) > 0$ for all $y \geq x'$. Since $S$ is not  $\P_x$-a.s.~bounded, there is $k \geq  0$ such that $\P_x[S_k \geq x'] > 0$. Then by the Markov property of $S$, the choice of $x'$ and dominated convergence,
	\begin{align*}
		&\lim_{n\to\infty}\E_x[(S_n-S_\infty)\1_{\{x' \leq S_k\leq S_{k+1}\leq\dots\leq S_n\}}] \\
&\qquad= \E_{x}\left[\lim_{n \to \infty} \E_{S_k}[(S_n-S_\infty)\1_{\{S_k\leq S_{k+1}\leq\dots\leq S_n\}}] \1_{\{S_k \geq x'\}}\right] \\
		&\qquad= \E_x\left[M_S(S_k)\1_{\{S_k \geq x'\}}\right] > 0.
	\end{align*}
	Thus, $S$ is a bubble under $\P_x$ by Corollary \ref{cor:char:equivalent}.
\end{proof}

In the remainder of this section, we seek to characterise the function $M_S$ in an analytic way and provide conditions for it to be non-zero.

\medskip{}
First, we show that $M_S$ solves a fixed point equation, more precisely a homogeneous \emph{Volterra integral equation of the second kind}, cf.~Brunner \cite{brunner:17} for a textbook treatment.

\begin{lemma}
	\label{lem:default fn}
	The default function $M_S$ is a solution to the Volterra integral equation
	\begin{equation}\label{Mfixed}
		M_S(x)=\int_{[x,\infty)}  M_S(y)K(x,\diff y),\quad x>0.
	\end{equation}
\end{lemma}

\begin{proof}
	Fix $x>0$. Using the definition of $M_S$, dominated convergence and the Markov property of $S$, we obtain
	\begin{align*}
		\int_{[x,\infty)}M_S(y)K(x,\diff y)&=\E_x\left[M_S(S_1)\1_{\{S_1\geq x\}}\right] \\
&=\E_x\left[\lim_{n \to \infty}\E_{S_1}\left[(S_n-S_\infty)\1_{\{S_n\geq S_{n-1}\geq \dots\geq S_1\geq x\}}\right]\right]\\
		&=\lim_{n \to\infty}\E_x\left[(S_n-S_\infty)\1_{\{S_n\geq S_{n-1}\geq \dots\geq S_1\geq x\}}\right]=M_S(x). 
	\end{align*}
\end{proof}

Note that \eqref{Mfixed} is non-standard in that the domain is non-compact. Therefore, we cannot apply standard existence and uniqueness results for Volterra integral equations, cf.~\cite[Chapter 8]{brunner:17}. In fact, existence is anyway not an issue since the zero function always solves \eqref{Mfixed}. Since the bubble case corresponds to \eqref{Mfixed} having a non-zero (nonnegative) solution, we are actually interested in \emph{non-uniqueness}, i.e., the case that \eqref{Mfixed} has multiple nonnegative solutions. By homogeneity of \eqref{Mfixed}, we then always have infinitely many solutions and so it is clear that we need an additional condition to pin down the default function $M_S$.

It follows from the definition of $M_S$, that $M_S(x) \leq x$ for all $x > 0$. So we consider nonnegative solutions to \eqref{Mfixed} that are dominated by the identity. To this end, denote by $\cI$ all Borel-measurable functions $M: (0, \infty) \to [0, \infty)$ satisfying $M(x) \leq x$ for all $x > 0$. Using that 
\begin{align*}
0 \leq \int_{[x,\infty)} M(y)K(x,\diff y) \leq \int_{(0,\infty)} M(y)K(x,\diff y) \leq \int_{(0,\infty)} y K(x,\diff y) = x
\end{align*}
for all $M \in \cI$ and $x > 0$, we can define the map
$\cK : \cI \to \cI$ by
\begin{equation*}
	\cK(M)(x) = \int_{[x,\infty)} M(y)K(x,\diff y), \quad x > 0.
\end{equation*}
Then the nonnegative solutions to \eqref{Mfixed} dominated by the identity are precisely given by fixed points of $\cK$.

While the map $\cK$ is in general not a contraction (and therefore \eqref{Mfixed} may have multiple solutions on $\cI$), it is \emph{monotone}, and this property will prove crucial for our subsequent analysis.

\begin{proposition}
	\label{prop:cK mon}
	The map $\cK$ is monotone on $\cI$.
\end{proposition}

\begin{proof}
	Let $M_1, M_2 \in \cI$ with $M_1 \leq M_2$. Then monotonicity of the integral gives for $x > 0$,
	\begin{equation*}
		\cK(M_1)(x) = \int_{[x,\infty)} M_1(y)K(x,\diff y) \leq \int_{[x,\infty)} M_2(y)K(x,\diff y) =  	\cK(M_2)(x).
	\end{equation*}
\end{proof}

Due to monotonicity of $\cK$, it is very useful to consider \emph{subsolution} and \emph{supersolutions} to~\eqref{Mfixed} on~$\cI$.

\begin{definition}
	A function $M \in \cI$ is called a \emph{subsolution} to \eqref{Mfixed} if 
	\begin{equation*}
		M(x) \leq \int_{[x,\infty)} M(y)K(x,dy),\quad x>0.
	\end{equation*}
	It is called a \emph{supersolution} to \eqref{Mfixed} if
	\begin{equation*}
		M(x) \geq \int_{[x,\infty)} M(y)K(x,dy),\quad x>0.
	\end{equation*}
\end{definition}

The following result shows that we can construct from each sub- or supersolution a solution to \eqref{Mfixed} by Picard iteration. To this end, for $n \in \NN_0$, define $\cK^n(M)$ recursively  by $\cK^0(M) := M$ and $\cK^n(M) := \cK(\cK^{n-1}(M))$ for $n \geq 1$.

\begin{proposition}
	\label{prop:subsupersol}
	Let $M \in \cI$ be a sub- or supersolution to \eqref{Mfixed}. Then the limit $\cK^\infty(M) = \lim_{n \to \infty} K^n(M)$ exists and is a solution to \eqref{Mfixed}. Moreover, 
\begin{itemize}
\item if $M$ is a subsolution, then the sequence $(\cK^n(M))_{n \in \NN}$ is nondecreasing and $\cK^\infty(M)$ is the smallest solution dominating $M$;
\item if $M$ is a supersolution, then the sequence $(\cK^n(M))_{n \in \NN}$ is nonincreasing and $\cK^\infty(M)$ is the largest solution dominated by $M$.
\end{itemize}
\end{proposition}

\begin{proof}
	We only consider the case that $M$ is a subsolution; the proof for the case that $M$ is a supersolution is analogous. If $M$ is a subsolution, the sequence $(\cK^n(M))_{n \in \NN_0}$ is nondecreasing in $M$ by monotonicity of $\cK$. Hence the limit $\lim_{n \to \infty} K^n(M)$ exists and is in $\cI$ since each $\cK^n(M)$ is in $\cI$. Moreover, it follows from monotone convergence that
	\begin{align*}
		\cK^\infty(M)(x) &= \lim_{n \to \infty} \cK^n(M)(x)=  \lim_{n \to \infty} \int_{[x,\infty)} \cK^{n-1}(M)(y) K(x,\diff y) \\
		&= \int_{[x,\infty)} \cK^\infty(M)(y) K(x,\diff y), \quad x > 0,
	\end{align*}
	whence $\cK^\infty(M)$ is a solution to \eqref{Mfixed}.
	
	Moreover, let $\tilde{M} \in \cI$ be any solution to \eqref{Mfixed} dominating $M$. It suffices to show that $\tilde M \geq \cK^n(M)$ for all $n \in \NN_0$. We argue by induction. The induction basis is trivial. For the induction step, suppose that $n \geq 1$ and $\tilde M \geq \cK^{n-1}(M)$. Then by the induction hypothesis and the definition of $\cK^n(M)$, for $x > 0$,
	\begin{equation*}
		\tilde{M}(x)=\int_{[x,\infty)}\tilde{M}(y)K(x,\diff y)\geq \int_{[x,\infty)}\cK^{n-1}(M)(y)K(x,\diff y)=\cK^n(M)(x).
	\end{equation*}
\end{proof}

We note the following important corollary.

\begin{corollary}
	\label{cor:iteration} 
	The largest solution to \eqref{Mfixed} on $\cI$ is given by $\cK^\infty(\id)$, where $\id$ denotes the identity function.
\end{corollary}

It follows from Lemma \ref{lem:default fn} and Corollary \ref{cor:iteration} that the default function $M_S$ is dominated by $\cK^\infty(\id)$. Under a mild assumption on the kernel $K$, we can assert that $M_S$ coincides with $\cK^\infty(\id)$. Thus, in this case, we can characterise the default function $M_s$ as the maximal solution to \eqref{Mfixed} dominated by the identity.

\begin{theorem}\label{thm:M max solution}
	Suppose that Assumption \ref{ass:lowerbound:a} is satisfied for any $x_a > 0$. Then $M_S$ is the maximal solution to \eqref{Mfixed} dominated by the identity. It is given  by $M_S = \cK^\infty(\id)$. Moreover, $M_S(x) <  x$ for all $x > 0$ and if $M_S \neq 0$, then 
	\begin{equation}
		\label{eq:thm:M max solution}
		\limsup_{x \to \infty} \frac{M_S(x)}{x} = 1.
	\end{equation}
\end{theorem}
\begin{proof}
	We first show by induction that for each $n \in \NN_0$,
	\begin{equation*}
		\cK^n(\id)(x) = \E_x\left[S_n\1_{\{S_n\geq S_{n-1}\geq \dots\geq S_1\geq S_0\}}\right], \quad x > 0.
	\end{equation*}
	The induction basis $n = 0$ follows from the martingale property of $S$. For the induction step, let $n \geq 1$. By the definition of $\cK^n(\id)$ and the Markov property of $S$, we obtain
	\begin{align*}
		\cK^n(\id)(x) &=	\int_{[x,\infty)}\cK^{n-1}(\id)(y)K(x,\diff y)=\E_x\left[\cK^{n-1}(\id)(S_1)\1_{\{S_1\geq x\}}\right]\\
		&=\E_x\left[\E_{S_1}\left[S_n\1_{\{S_n\geq S_{n-1}\geq \dots\geq S_1\}}\right]\1_{\{S_1\geq x\}}\right] \\ &=\E_x\left[S_n\1_{\{S_n\geq S_{n-1}\geq \dots\geq S_1\geq S_0\}}\right].
	\end{align*}
	Hence, the definitions of $\cK^\infty(\id)$ and $M_S$ together with \eqref{eq:prop:Markov:cond a} give
	\begin{align*}
		\cK^\infty(\id)(x) &= \lim_{n \to \infty}\E_x\left[S_n\1_{\{S_n\geq S_{n-1}\geq \dots\geq S_1\geq S_0\}}\right] \\
		&=  \lim_{n \to \infty}\E_x\left[(S_n - S_\infty)\1_{\{S_n\geq S_{n-1}\geq \dots\geq S_1\geq S_0\}}\right] = M_S(x), \quad x > 0.
	\end{align*}
	It follows from Corollary \ref{cor:iteration} that $M_S$ is the maximal solution to \eqref{Mfixed} dominated by the identity. 
	
	Moreover, $M_S \leq \id$, the fact that $\P_x[S_1 <x] = a(x) > 0$ for all $x > 0$ and the martingale property of $S$ give
	\begin{align*}
		M_S(x)=\E_x\big[M_S(S_1)\1_{\{S_1\geq x\}}\big]\leq\E_x\big[S_1\1_{\{S_1\geq x\}}\big]<\E_x\big[S_1\big]=x.
	\end{align*}

	Finally, to establish \eqref{eq:thm:M max solution}, denote the function $H_S: (0, \infty) \to [0, 1]$ by $H_S(x) = \sup_{y \geq x} \tfrac{M_S(x)}{x}$. If $M_S \neq 0$, it follows from Proposition \ref{prop:MS:zero} that $H_S(x) >0$ for all $x > 0$. 	It suffices to show that $H_S(x) = 1$ for all $x > 0$. Seeking a contradiction, suppose there exists $x' > 0$ such that $H_S(x') < 1$. Define the function $M: (0, \infty) \to [0, \infty)$ by 
	\begin{equation*}
	M(y) =
	\begin{cases}
0  & \text{if } x < x', \\
\frac{M_S(x)}{H_S(x')} & \text{if } x \geq x'.
	\end{cases}
	\end{equation*}
Then $M \in \cI$ by the fact that $H_S(x') \geq \tfrac{M_S(x)}{x}$ for $x \geq x'$. Moreover, $M$ is a subsolution to \eqref{Mfixed} and $M(x) > M_S(x)$ for $x \geq x'$. This together with Proposition \ref{prop:subsupersol} implies that $\cK^\infty(M)$ is a solution to  \eqref{Mfixed}  dominated by the identity. Since $\cK^\infty(M)(x) \geq M(x) > M_S(x)$ for $x > x'$, this is in contradiction to $M_S$ being the maximal solution to \eqref{Mfixed}  dominated by the identity.
\end{proof}

The following corollary shows that if we can find a non-trivial subsolution to \eqref{Mfixed}, then $S$ is a bubble.

\begin{corollary}\label{cor:existence}
	Suppose Assumption \ref{ass:lowerbound:a} is satisfied for any $x_a > 0$. If $M \in \cI$ is a subsolution to \eqref{Mfixed}, then $M \leq M_S$. If in addition $\liminf_{x \to \infty} M(x) > 0$, then $S$ is a bubble under $\P_x$ for all $x > 0$ for which $S$ is not $\P_x$-a.s. bounded.
\end{corollary}

\begin{proof}
Proposition \ref{prop:subsupersol} and Theorem \ref{thm:M max solution}  give $M \leq \cK^\infty(M) \leq M_S$. The additional claim then follows from Proposition \ref{prop:MS:liminf}.
\end{proof}

A typical candidate for a subsolution in Corollary \ref{cor:existence} is the call function $M(x) = (x -L)^+$ for some $L > 0$.  This is illustrated by the following example. Note that this example cannot be addressed with the results from Section \ref{sec:Markov}.
\begin{example}
	Suppose that for all $x>0$, $K(x,\diff y)=k(x,y) \dd y$, where the density $k$ satisfies
	\[k(x,y)=\frac{2}{3(x+1)}\1_{[x, 2x]}(y)\quad \text{for }y\geq x>0,\]
	and $k(x, y)$ on $\{(x,y):0<y<x\}$ is chosen such that $K$ is a martingale kernel.
Then the function $a$ from \eqref{eq:Markov:a} satisfies $a(x) = 1 - \int_{[x, \infty)} k(x, y) \dd y = \tfrac{3+x}{3 + 3x} \geq \frac{1}{3}$ so that Assumption \ref{ass:lowerbound:a} is satisfied for any $x_a > 0$.
	Consider $M(x):=(x-3)^+$. Then trivially $	\int_x^\infty  M(y)k(x,y)\dd y \geq 0 = M(x)$ for $x \leq 3$ and 
	\begin{equation*}
		\int_x^\infty M(y) k(x,y) \dd y= x - 3\frac{x}{x+1} \geq x-3 = M(x) \text{ for } x > 3.
	\end{equation*}
	It follows that $M$ is subsolution to \eqref{Mfixed} and we may deduce that $S$ has a  bubble under $P_x$ for all $x > 0$ by Corollary \ref{cor:existence} since $S$ is not $\P_x$-a.s.~bounded for any $x > 0$.
\end{example}

While Theorem \ref{thm:M max solution} provides a characterisation of the default function $M$, it does not provide a criterion to decide whether \eqref{Mfixed} has a non-trivial, i.e., a non-zero nonnegative solution dominated by the identity. Moreover, it does not provide a criterion to decide whether a given candidate solution $M$ to \eqref{Mfixed} is indeed maximal.  Under a stronger assumption on the kernel $K$,  we can provide a sufficient criterion for the existence of non-trivial solutions to \eqref{Mfixed} dominated by the identity. Moreover, we obtain a local uniqueness result in this case. To this end, recall the definitions of the functions $a$ and $b$ from \eqref{eq:Markov:a} and \eqref{eq:Markov:b}, respectively. Moreover, denote by $\Vert \cdot \Vert_{\sup}$ the supremum norm.
\begin{theorem}
	\label{thm:contraction}
	Suppose $\inf_{x > 0} a(x) > 0$.  Then the following are equivalent:
	\begin{enumerate}
		\item [(a)] $\sup_{x > 0} x b(x) < \infty$.
		\item [(b)] For all $L > 0$ sufficiently large, the call function $M(x) = (x - L)^+$ is a subsolution to \eqref{Mfixed}.
		\item [(c)] The default function $M_S$ is non-trivial and satisfies $\Vert \id - M_S\Vert_{\sup} < \infty$.
	\end{enumerate}
	Moreover, if one of the above conditions is satisfied, $M_S$ is the unique solution to \eqref{Mfixed} among all solutions $M \in \cI$ satisfying $\Vert  \id - M \Vert_{\sup} < \infty$.
\end{theorem}

\begin{proof}
	(a) $\Rightarrow$ (b). Set $\alpha := \inf_{x > 0} a(x) > 0$, $\beta := \sup_{x > 0} x b(x)$ and $L \geq \frac{\beta}{\alpha}$. Then the call function $M(x) :=  (x - L)^+$ satisfies $	\int_{[x, \infty)} M(y)K(x, \diff y)\geq 0 = M(x)$ for $x \leq L$ and 
	\begin{align*}
		\int_{[x, \infty)} M(y) K(x, \dd y) &= x (1 - b(x)) - L(1- a(x)) = M(x)  - \beta(x) x + \alpha(x) L \\
		&\geq M(x) -\beta + \alpha L \geq M(x), \quad x > L.
	\end{align*}
	Hence, $M$ is a subsolution to \eqref{Mfixed}.
	
	(b) $\Rightarrow$ (c). Let $L > 0$ such that $M(x) = (x - L)^+$ is a subsolution to \eqref{Mfixed}. Then $M_S \geq M$ by Corollary \ref{cor:existence}, whence $M_S$ is non-trivial and satisfies $\Vert \id - M_S\Vert_{\sup} < \Vert \id - M\Vert_{\sup}  = L$.
	
	(c) $\Rightarrow$ (a). Since $\cK(\id) \geq \cK^\infty(\id) = M_S$ by Proposition \ref{prop:subsupersol} and Theorem \ref{thm:M max solution}, it follows that $\Vert \id - \cK(\id)  \Vert_{\sup} \leq \Vert \id - M_S \Vert_{\sup} < \infty$. Now the claim follows from the fact that $x b(x) = \int_{(0, x)} y K(x, \diff y) =  x - \cK(\id)(x)= \id(x) - \cK(\id)(x)$ for $x > 0$.
	
	For the additional claim, set $\cI_{\sup} := \{M \in \cI: \Vert \id - M \Vert_{\sup}< \infty \}$. Then $\cI_{\sup}$ is a complete metric space for the metric generated by the supremum norm. Moreover, $\cK$ maps  $\cI_{\sup}$ to itself. Indeed, for each $M \in \cI_{\sup}$, by (b), there exists $L \geq \Vert \id - M \Vert_{\sup}$ such that $\tilde M (x) := (x-L)^+$ is a subsolution to \eqref{Mfixed}. Hence by monotonicity of $\cK$ and the fact that $\tilde M \leq M $ is a subsolution, 
	\begin{equation*}
		\Vert \id - \cK(M) \Vert_{\sup} \leq 	\Vert \id - \cK(\tilde M) \Vert_{\sup}  \leq \Vert \id - \tilde M \Vert_{\sup} = L < \infty.
	\end{equation*}
	Finally, we show that $\cK$ is a contraction on $\cI_{\sup}$. Let $M_1, M_2 \in \cI_{\sup}$. Then 
	\begin{align*}
		|\cK(M_1)(x) - \cK(M_2)(x)| &= \bigg\vert\int_{[x, \infty)} (M_1(y)- M_2(y)) K(x, \diff y)\bigg\vert \\
&\leq \Vert M_1 - M_2 \Vert_{\sup} \int_{[x, \infty]} K(x, \diff y) \\
		&\leq (1 - \alpha)\Vert M_1 - M_2 \Vert_{\sup}.
	\end{align*}
	Taking the supremum over $x$ shows that $\cK$ is indeed a contraction since $\alpha > 0$. Now Banach's fixed point theorem implies that $\cK$ contains a unique fixed point, and by (c), this fixed point is $M_S$.
\end{proof}

We proceed to illustrate Theorem \ref{thm:contraction} by an example.

\begin{example}\label{ex:fixedpoint} Suppose that for all $x>0$, $K(x,\diff y)=k(x,y) \dd y$, where the density satisfies
	\[k(x,y)=\frac{e}{2}\cdot\frac{1-e^{-x}}{1-e^{-y}} \frac{1}{x} e^{-y/x} \quad \text{for }y\geq x>0,\]
	and $k(x, y)$ on $\{(x,y):0<y<x\}$ is chosen such that $K$ is a martingale kernel.
Note that 
	\begin{equation*}
		\int_x^\infty k(x,y) \dd y < \frac{e}{2} \int_x^\infty  \frac{1}{x} e^{-y/x} \dd y = \frac{1}{2}.
	\end{equation*}
	This implies in particular that $a(x) \geq 1/2$ for all $x > 0$. In this case, the fixed point equation \eqref{Mfixed} is given by
\begin{align*}
	M(x) &=\int_x^\infty M(y)\cdot\frac{e}{2}\cdot\frac{1-e^{-x}}{1-e^{-y}}\cdot\frac{1}{x}\cdot e^{-y/x} \dd y \notag \\
\Leftrightarrow\quad\frac{M(x)x}{1-e^{-x}} &=\frac{e}{2}\int_x^\infty \frac{M(y)}{1-e^{-y}}\cdot e^{-y/x} \dd y.
	\end{align*}
	As can easily be checked, the function $M_\lambda(x) := \lambda x(1-e^{-x})$, $x > 0$, is a solution  in $\cI$ to \eqref{Mfixed}  for any $\lambda \in [0, 1]$. As $M_S$ is the largest solution to \eqref{Mfixed} dominated by the identity, this yields the candidate $M_1(x) =x(1-e^{-x})$. Since $\Vert \id - M_1 \Vert_{\sup} = e^{-1} < \infty$, it follows from Theorem  \ref{thm:contraction}  that $M_S = M_1$.
\end{example}

Combing Theorem \ref{thm:contraction} with Proposition \ref{prop:MS:liminf}, we get the following existence results for bubbles. Note that this result covers cases that cannot be treated with the theory of Section \ref{sec:Markov}.
\begin{corollary}
	Suppose that $\inf_{x > 0} a(x) > 0$ and $\sup_{x > 0} x b(x) < \infty$. Then $S$ is a bubble under $\P_x$ for all $x > 0$ for which $S$ is not $\P_x$-a.s. bounded.
\end{corollary}


\section{Relation to the strict local martingale definition of asset price bubbles in continuous time models}\label{ch:relation}
 
In this final section, we discuss how our definition of  bubbles in discrete time relates to the strict local martingale definition of bubbles in continuous time. To approach this question, one first has to discretise a positive continuous local martingale $X = (X_t)_{t \geq 0}$ in continuous time in such a way that it becomes a discrete time martingale.  Of course, there are many ways to do this and we choose a somewhat canonical construction. To wit, we consider localising sequences $(\tau_n)_{n \in \NN}$ of stopping times with $\tau_n\to\infty$ $\p$-a.s.~such that for each $n$, $\tau_n$ and the stopped process $X^{\tau_n}$ are uniformly bounded. We then define the discrete time process $S = (S_n)_{n \in \NN}$ by $S_n : = X_{\tau_n}$. Then $S$ is a martingale by the stopping theorem and satisfies $S_\infty = X_\infty$ $\as{\p}$, which implies that $S$ is uniformly integrable if and only if $X$ is uniformly integrable.

The simplest way to get localising sequences as above is to choose two increasing sequences of positive real numbers $a=(a_n)_{n \in \NN}$ and $b=(b_n)_{n \in \NN}$ converging to infinity and to define the sequence $(\tau^{a,b}_n)_{n \in \NN}$ of stopping times by $\tau^{a,b}_0:=0$ and
\begin{equation}
	\label{eq:tau ab}
\tau^{a,b}_n:=\inf\{t\geq0:\ X_t \geq b_n\}\wedge a_n, \quad n \in \NN.
\end{equation}
Then $(\tau^{a,b}_n)_{n\in\N}$ is a localizing sequence of stopping times for $X$ with $\tau^{a, b}_n \leq a_n$ and $\sup_{t \geq 0} X^{\tau^{a, b}_n}_t \leq b_n$.

In the special case that $X$ is a Markov process, we would like to stop in such a way that the discrete time process $S$ is again a Markov process. In this case, the simplest way to get localising sequences as above is to choose two constants $\alpha, \beta > 0$ and to define the sequence of stopping times $(\tau^{\alpha,\beta}_n)_{n \in \NN}$ by $\tau^{\alpha,\beta}_0:=0$ and
\begin{equation}
		\label{eq:tau alphabet}
	\tau^{\alpha,\beta}_n:=\inf\left\{t\geq \tau^{\alpha, \beta}_{n - 1}:\ X_t \geq (1+\beta) X_{\tau^{\alpha,\beta}_{n-1}}\right\} \wedge \left(\tau^{\alpha,\beta}_{n-1}+ \alpha\right), \quad n \in \NN.
\end{equation}
In this case it is still true that  $(\tau^{\alpha,\beta}_n)_{n\in\N}$ is a localizing sequence of stopping times for $X$ and $\tau^{\alpha, \beta}_n$ and $X^{\tau^{\alpha, \beta}_n}$ are uniformly bounded.

Our first goal in this section is to show that if $X$ is a continuous positive strict local martingale, then the discrete time process $S$ is a bubble for either choice of stopping times.

The proof of this result relies on the following deep characterisation of strict local martingales in continuous time; cf.~\cite{Meyer,DelbaenSchachermayer,KKN}. Let $X = (X_t)_{t \geq 0}$ be a positive c\`adl\`ag local $\P$-martingale with $X_0 = x$. Then under some technical assumptions on the probability space and the underlying filtration, there exists a probability measure $\q$ with $\q|_{\F_t}\gg\p|_{\F_t}$ for all $t\geq0$ such that $Y:=1/X$ is a nonnegative true $\q$-martingale, and for all bounded stopping times $\tau$ and $A\in\cF_\tau$,
\[\p[A]=x\cdot\E^\q\left[Y_\tau \1_{A}\right].\]
Especially, we have the identity 
 \[\E^\p[X_0-X_t] =x\cdot\q[Y_t=0], \quad t \geq 0,\]
 i.e.,~$X$ is a strict local martingale on $[0,t]$ if and only if $\q[Y_t=0]>0$.
 
 With this, we have the following two results. 
 
 \begin{proposition}
 	\label{prop:tau ab}
 	Let $X=(X_t)_{t\geq0}$ be a continuous positive strict local $\p$-martingale. Let $a=(a_n)_{n \in \NN}$ and $b=(b_n)_{n \in \NN}$ be increasing sequences of positive real numbers converging to $\infty$. Define the sequence of stopping times  $(\tau^{a,b}_n)_{n \in \NN}$ by \eqref{eq:tau ab} and set $S^{a,b}_n = X_{\tau^{a,b}_n}$ for $n \in \NN_0$. Then the measure $\p$ is a bubble measure for the discrete time martingale $S^{a,b}=(S^{a,b}_n)_{n\in\N_0}$.
 \end{proposition}

 \begin{proposition}
	\label{prop:tau alphabet}
	Let $X=(X_t)_{t\geq0}$ be a continuous-time positive strict local Markov martingale under the measure $\P_{x}$. Let $\alpha, \beta > 0$ and define the sequence of stopping times  $(\tau^{\alpha,\beta}_n)_{n \in \NN}$ by \eqref{eq:tau alphabet} and set $S^{\alpha,\beta}_n = X_{\tau^{\alpha,\beta}_n}$ for $n \in \NN_0$. Then the measure $\p_{x}$ is a bubble measure for the discrete time Markov martingale $S^{\alpha,\beta}=(S^{\alpha,\beta}_n)_{n\in\N_0}$.
\end{proposition}

We only establish the proof of Proposition \ref{prop:tau ab}. The proof of Proposition \ref{prop:tau alphabet} is similar and left to the reader.
 
 \begin{proof}[Proof of Proposition \ref{prop:tau ab}]
Since $X$ is a local martingale with respect to its natural filtration and $(\tau^{a, b}_n)_{n \in \NN}$ is adapted to this filtration, we may assume without loss of generality that $X$ is the canonical process on $C([0, \infty),(0, \infty])$ with $X_0 = x > 0$. Set $\tau_\infty := \inf \{t \geq 0 : X_t = \infty\}$. Then $\tau^{a, b}_n < \tau_{\infty}$ for all $n \in \NN$.
 
	Then there exists a measure $\QQ$ on $C([0, \infty),(0, \infty])$, with $\q|_{\F_t}\gg\p|_{\F_t}$ for all $t\geq0$, such that $Y := 1/X$ is a nonnegative $\QQ$-martingale and $\EX[\P]{X_\tau \1_A} = x_0\cdot \QQ[A]$ for each bounded stopping time $\tau < \tau_{\infty}$ and $A \in \cF_\tau$. Let $k:=\min\{n\in\N: \E[X_0-X_{a_n}]>0\}$. Then for $n\geq k$, using that $\tau^{a,b}_n$ is bounded by $a_n$ and $\tau^{a,b}_n < \tau_{\infty}$, we obtain
 	Taking the limit as $n\to\infty$ on the left hand side, it follows that $\p$ is a bubble measure for $S$ by Theorem \ref{thm:char:equivalent}.
 	\end{proof}
 
 We proceed to illustrate  Proposition \ref{prop:tau alphabet} by an Example.
 
 \begin{example}
 	Let $(X_t)_{t\geq0}$ be the three-dimensional inverse Bessel process, i.e., $X$ is the unique strong solution to the SDE $\dd X_t=-X_t^2 \diff B_t$, where $B = (B_t)_{t \geq 0}$ is a $\p_x$-Brownian motion. The process $Y:=1/X$ is $\q_{1/x}$-Brownian motion, stopped when it reaches zero.
 	
 	Fix $\alpha, \beta > 0$. We proceed to calculate the Markov kernel $K(x, \dd y)$ for $S^{\alpha, \beta}$ under $\P_{x}$. Denote by $W$ a standard Brownian motion starting at zero and by $\Phi$ the cdf of standard normal random variable. Using the reflection principle of Brownian motion and denoting the running supremum and the running infimum of a process $Z$ by $\ol Z$ and $\ul Z$, respectively, we obtain
 \begin{align*}
&\p_x\left[S^{\alpha,\beta}_1 = (1+\beta)x \right] =\p_x\left[X_{\tau^{\alpha, \beta}_1}= (1+\beta)x\right] \\
&\qquad\qquad=x \E^\q_{1/x}\left[Y_{\tau^{\alpha,\beta}_1}\1_{\{Y_{\tau^{\alpha,\beta}_1}=\frac{1}{(1+\beta) x}\}}\right]=\frac{1}{1+\beta}\q_{1/x}\left[\ul {Y}_{\alpha} \leq \frac{1}{(1+\beta) x}\right] \\
&\qquad\qquad= \frac{1}{1+\beta}\q\left[\ol W_\alpha \geq \frac{1}{x} - \frac{1}{(1+\beta) x}\right] 
= \frac{2}{1+\beta} \Phi\left(- \frac{\beta}{(1+\beta) x \sqrt{\alpha}} \right).
	\end{align*}
Moreover, for $z \in (0, (1+\beta)x)$, using the joint density of Brownian motion and its running supremum, we obtain
	\begin{align*}
 		&\p_x\left[S^{\alpha,\beta}_1 \leq z \right] =\p_x\left[X_{\tau^{\alpha, \beta}_1} \leq z\right] =\p_x\left[X_\alpha \leq z, \ol X_\alpha < (1+\beta) x\right] \\
&=x \E^\q_{1/x}\left[Y_{\alpha}\1_{\{Y_{\alpha} \geq \frac{1}{z}\}}  \1_{\{\ul Y_{\alpha} > \frac{1}{(1+\beta)x}\}}\right]\\
 	&=x \E^\q \left[\left(\frac{1}{x} - W_\alpha\right)\1_{\{W_{\alpha} \leq \frac{1}{x} - \frac{1}{z}\}} \1_{\{\ol W_{\alpha} < \frac{1}{x} -\frac{1}{(1+\beta)x}\}} \right]\\
 	&= x \int_{-\infty}^{\frac{1}{x} - \frac{1}{z}} \left(\frac{1}{x} - y\right)\int_0^\frac{\beta}{(1+\beta)x}  \frac{2(2u-y)}{\sqrt{2\pi \alpha^3}}\exp\left(-\frac{(2u -y)^2}{2\alpha}\right)   \dd u \dd y \\
	&= x \int_{-\infty}^{\frac{1}{x} - \frac{1}{z}} \left(\frac{1}{x} - y\right)\int_{\frac{-y}{\sqrt{\alpha}}}^{\frac{\frac{2\beta}{(1+\beta)x}-y}{\sqrt{\alpha}}}  \frac{v}{\sqrt{2\pi \alpha}}\exp\left(-\frac{v^2}{2}\right)   \dd v \dd y \\
		&= x \int_{-\infty}^{\frac{1}{x} - \frac{1}{z}} \frac{1}{\sqrt{2\pi \alpha}}\left(\frac{1}{x} - y\right) \bigg(\exp\left(-\frac{y^2}{2 \alpha} \right) -\exp\bigg( -\frac{(\frac{2\beta}{(1+\beta)x}-y)^2}{2 \alpha} \bigg) \bigg) \dd y\\
 &= x \int_{0}^{z} \frac{1}{\sqrt{2\pi \alpha} w^3} \bigg(\exp\bigg(-\frac{(\frac{1}{x} - \frac{1}{w})^2}{2 \alpha} \bigg) -\exp\bigg( -\frac{(\frac{2\beta}{(1+\beta)x}-\frac{1}{x} + \frac{1}{w})^2}{2 \alpha} \bigg) \bigg) \dd w
\end{align*}
 
 Thus, the Markov kernel $K(x, \diff y)$ for $S^{\alpha, \beta}$ under $\P_{x}$ is given by
 	\begin{align*}
 	K(x, \diff y) &= \frac{2}{1+\beta} \Phi\left(- \frac{\beta}{(1+\beta) x\sqrt{\alpha}} \right) \delta_{(1+\beta) x}(\diff y) \\
	&+ \frac{x}{\sqrt{2\pi \alpha} y^3} \bigg(e^{-\frac{(\frac{1}{y} - \frac{1}{x})^2}{2 \alpha}}-e^{-\frac{(\frac{1}{y} +\frac{\beta-1}{(1+\beta)x} )^2}{2 \alpha} } \bigg)  \1_{(0, (1+\beta) x)}(y) \dd y.
	\end{align*}
 This ends the example.
 \end{example}
 
 We finish this section by providing a converse to Proposition \ref{prop:tau ab}.
\begin{theorem}
	\label{thm:slm:char}
 Let $X=(X_t)_{t\geq0}$ be a continuous positive local $\p$-martingale that converges to zero $\as{\P}$ Then $X$ is a strict local $\p$-martingale if and only if for all sequences $a=(a_n)_{n \in \NN}$ and $b=(b_n)_{n \NN}$ converging to infinity, the measure $\p$ is a bubble measure for the discrete time martingale $S^{a,b}=(S^{a,b}_n)_{n\in\N_0}$.
\end{theorem}
 
 \begin{proof}
 First, assume that $X$ is a strict local martingale. The the result follows from Proposition \ref{prop:tau ab}

 Conversely, suppose that $X$ is a true $\P$-martingale. As in the proof of Proposition \ref{prop:tau ab}, we may assume without loss of generality that $X$ is the canonical process on $C([0, \infty),(0, \infty])$ with $X_0 = 1$. Then there exists a measure $\QQ$ on $C([0, \infty),(0, \infty])$, with $\q|_{\F_t}\gg\p|_{\F_t}$ for all $t\geq0$, such that $Y := 1/X$ is a positive $\QQ$-martingale that converges to $0$ $\QQ$-almost surely and $\EX[\p]{X_\tau \1_A} = \QQ[A]$ for each bounded stopping time $\tau < \tau_{\infty}$ and $A \in \cF_\tau$. 
 By Proposition \ref{prop:mart:decreasing}, there exists an increasing sequence $a=(a_n)_{n \in \NN}$ converging to infinity such that 
 for each $k \in \NN$,
 \begin{equation}
 	\label{eq:pf:thm:slm:char:a}
 \q[Y_{a_k} \geq Y_{a_{k+1}} \geq \cdots] = 0.
 \end{equation}
 Since $Y$ is positive $\as{\QQ}$, we can find an increasing sequence $b=(b_n)_{n \in \NN}$ converging to infinity such that $\q[\ul{Y}_{a_n}\leq1/b_n]<2^{-n}$ for all $n\in\N$. By the Borel--Cantelli lemma, this implies that
 \begin{equation}
 	 	\label{eq:pf:thm:slm:char:b}
\q\left[\ul{Y}_{a_k}\leq\frac{1}{b_k}\text{ for infinitely many }k\right]=0.
 \end{equation}
Then for any $k \in \NN$, using \eqref{eq:pf:thm:slm:char:a} and \eqref{eq:pf:thm:slm:char:b},and recalling that  each $\tau^{a,b}_n$ is bounded by $a_n$, we obtain
 \begin{align*}
&\lim_{n\to\infty}\E^\p\left[S^{a,b}_n\1_{\{S^{a,b}_k \leq S^{a,b}_{k+1} \leq \cdots \leq S^{a,b}_n\}} \right] = \lim_{n\to\infty} \q\left[Y_{\tau^{a,b}_n}\leq Y_{\tau^{a,b}_{n-1}}\leq \cdots\leq Y_{\tau^{a,b}_k}\right] \\
&\qquad=  \q\left[Y_{\tau^{a,b}_k}\geq Y_{\tau^{a,b}_{k +1}} \geq \cdots \right] \\
	&\qquad\leq\q\left[\ul{Y}_{a_\ell}\leq\frac{1}{b_\ell}\text{ for infinitely many }\ell\right] +\q\left[ Y_{a_\ell}\leq Y_{a_{\ell-1}}\text{ eventually}\right] = 0.
\end{align*}
 By Theorem \ref{thm:char:equivalent} , this shows that $\p$ is not a bubble measure for $S^{a,b}$.
\end{proof}

\appendix
\section{Auxiliary results}

\begin{proposition}
	\label{prop:mart:decreasing}
Let $(M_t)_{t \geq 0}$ be a  positive continuous local martingale that $\as{\p}$ never becomes constant, i.e.~$\p(M_t=M_\infty \text{ for all } t\geq s)=0$ for all $s\geq0$. Then there exists an increasing sequence of nonnegative real numbers $(a_k)_{k\in \NN}$ with $\lim_{k \to \infty} a_k = \infty$ such that for each $k \in \NN$,
\begin{equation*}
P[M_{a_k} \geq M_{a_{k+1}} \geq \cdots] = 0.
\end{equation*}
\end{proposition}

\begin{proof}
For $n, m \in \NN_0$, set $D^n_m := \{n + j 2^{-m}: j \in \NN_0\}$. Then for each fixed $n \in \NN_0$, $D^n_m$ is increasing in $m$. Set $D^n_\infty := \bigcup_{m\in \NN_0} D^n_m$, which is dense in $[n,\infty)$. Since $M$ is a continuous local martingale that $\as{\p}$  never becomes constant, its paths $\as{\p}$ never become monotone. By continuity of the paths of $M$, this implies that for each $n \in \NN_0$,
\begin{equation*}
\P[M_s \geq M_t \text{ for all }   s, t\in D^n_\infty \text{ with } s < t] = 0.
\end{equation*}
By $\sigma$-continuity of $\P$, for each $n \in \NN_0$, there exists $m_n \in \NN_0$ such that
\begin{equation*}
	\P[M_s \geq M_t \text{ for all }   s,t \in D^n_{m_n} \text{ with } s < t] \leq 2^{-n}.
\end{equation*}
We may assume without loss of generality that the sequence $(m_n)_{n \in \NN}$ is nondecreasing. Define the set $D$ by
\begin{equation*}
D := \bigcup_{n \in \NN_0} \left(D^n_{m_m} \cap [n, n+1)\right).
\end{equation*}
Then $D \cap [n, \infty) \supset D^{n'}_{m_{n'}}$ for all $n' \geq n$. This implies that
for each $k \in \NN_0$,
\begin{equation*}
	\P\left[M_s \geq M_t \text{ for all }   t, s \in D \cap [n, \infty) \text{ with } s < t\right] \leq\lim_{n' \to \infty} 2^{-n'} =0.
\end{equation*}
If $(a_k)_{k \in \NN}$ is an enumeration of $D$ in increasing order, the result follows.
\end{proof}

\begin{remark}
Note that in the situation of Proposition \ref{prop:mart:decreasing}, in general there also exist sequences $(a_k)_{k\in \NN}$ with $\lim_{k \to \infty} a_k = \infty$ such that
\begin{equation*}
	P[M_{a_1} \geq M_{a_2} \geq \cdots] > 0.
\end{equation*}
For example let $M_t = \exp(W_t - t/2)$, where $W$ is a Brownian motion. Define the sequence  $(a_k)_{k\in \NN}$ by $a_1 = 1$ and $a_k = a_{k-1} + k$. Then using that Brownian motion has independent and normally distributed increments, denoting the cdf of  a standard normal random variable by $\Phi$
, we obtain
\begin{equation*}
	\P[M_{a_1} \geq M_{a_2} \geq \cdots] = \prod_{k =1}^\infty \Phi\left(\frac{\sqrt{k}}{2}\right)  \geq  \prod_{k =1}^\infty\left (1 - \frac{1}{\sqrt{2\pi} \frac{\sqrt{k}}{2}} \exp \left(-\frac{1}{8} k\right)\right) > 0.
\end{equation*}
\end{remark}

\end{document}